\documentclass[11pt]{article}
\makeatletter

\usepackage[english]{babel}

\usepackage[utf8]{inputenc}

\usepackage{tikz}
\usetikzlibrary{matrix,arrows,decorations.pathmorphing}
\usepackage{tikz-cd}

\usepackage{lmodern}
\usepackage[T1]{fontenc}
\usepackage{amssymb}
\usepackage{latexsym}
\usepackage{mathtools}
\usepackage{hyperref}
\usepackage{verbatim}
\usepackage{authblk}

\usepackage{amsthm}
\theoremstyle{plain}

\newtheorem{thm}{Theorem}[section]
 
\newtheorem{cor}[thm]{Corollary}
\newtheorem{lem}[thm]{Lemma}
\newtheorem{prop}[thm]{Proposition}

\theoremstyle{definition}
\newtheorem{rem}[thm]{Remark}

\newtheorem{exmp}[thm]{Example}

\def\gcd{\mathrm{GCD}}

\def\deg{\mathrm{deg}}
\def\aut{\mathrm{Aut}}
\def\div{\mathrm{div}}
\def\ord{\mathrm{ord}}

\def\cX{\mathcal{X}}
\def\cE{\mathcal{E}}
\def\K{\mathbb{K}}
\def\gg{g(\mathcal{X})}

\title{Curves with a large automorphism group admitting a cyclic subgroup of index $2$}
\author[1]{Arianna Dionigi}
\author[2]{Massimo Giulietti}
\author[2]{Marco Timpanella}

\affil[1]{%
Department of Mathematics Ulisse Dini,
University of Florence,
Via Giovanni Battista Morgagni 67/a,
50134 Florence, Italy}

\affil[2]{%
Department of Mathematics and Computer Science,
University of Perugia,
Via Luigi Vanvitelli 1,
06123 Perugia, Italy}
\date{}
\begin{document}
\maketitle

\begin{abstract}
The Hurwitz bound on the order of the $\mathbb K$-automorphism group ${\rm{Aut}}({\mathcal{X}})$ of an algebraic curve ${\mathcal{X}}$ of genus $g(\mathcal{X})\ge 2$ 
defined over a field $\mathbb K$ of zero characteristic states that $|{\rm{Aut}}({\mathcal{X}})|\le 84(g(\mathcal{X})-1)$.
Improved bounds are available for the order of certain types of subgroups within automorphism groups.
 For instance, if a subgroup $H$ of ${\rm{Aut}}({\mathcal{X}})$ is dihedral, then in the complex case, $|H| \leq 4g(\mathcal{X}) + 4$. More recently it has been shown that a tighter bound holds for $H$ a generalized quasi-dihedral group.
 In this paper we explore the more general setting of a curve defined over a field of any characteristic, and $H$ a group admitting a cyclic subgroup of index two. We show that the same upper bound for the size of a dihedral group of automorphisms holds for curves defined over an algebraically closed field of characteristic $p\ne 2$. Then we provide some classification results about (non-dihedral) groups of size larger than $4g(\cX)+4$ admitting a cyclic subgroup of index $2$.
\end{abstract}


\section{Introduction}

In this paper, let $\mathcal{X}$ denote a projective, geometrically irreducible, nonsingular algebraic curve defined over an algebraically closed field $\mathbb{K}$ with characteristic $p$. We denote the automorphism group of $\mathcal{X}$ that fixes $\mathbb{K}$ elementwise by $\aut(\mathcal{X})$.
By a classical result, $\aut(\mathcal{X})$ is finite if the genus $g(\mathcal{X})$ of $\mathcal{X}$ is at least two; see, for example, \cite[Chapter 11]{HKT}. So, throughout the paper we assume that $\cX$ has genus $g(\cX)\ge 2$.

Generally speaking, numerous authors have studied curves over the complex field (or, equivalently, Riemann surfaces) that admit automorphism groups of specific type and of relatively large order compared to their genus. These investigations often involve classifying such surfaces and analyzing the structural conditions under which they can possess "large" symmetry groups. This typically requires tools from group theory, complex geometry, and topology. However, in the case of positive characteristic, the available methods are more limited and the approach must rely entirely on algebraic techniques.

If $H$ is a cyclic group acting on a curve $\mathcal X$, then $|H| \leq 4g(\mathcal{X}) + 2$. A complete characterization of algebraic curves  with an automorphism of order at least $2g(\mathcal{X}) + 1$ was given by Irokawa and Sasaki \cite{sasaki} in the complex case and by the authors  in the general case \cite{DGT}. 
In this paper, we investigate curves having an automorphism group $H$ admitting a cyclic subgroup $G$ of index $2$, over fields of any characteristic $p\neq 2$. Admitting a cyclic subgroup of index $2$ might seem a strong condition, but actually there are many possible structures for such groups. Apart from the cyclic, abelian, dihedral, generalized quasi-dihedral cases, there are several others, and a complete list can be found for instance in \cite[Proposition 10.2]{Noe}.

Clearly,  if an automorphism group $H$ admits a cyclic subgroup of index $2$,
then $|H|\le 8\gg+4$. 
Over the complex field, it is a classical result that if $H$ is dihedral, then $|H|\le 4\gg+4$; see e.g. \cite{BUJ2}. Recently,  it has been shown (\cite[Theorem 6.1]{HMQ}) that if $H$ is a generalized quasi-dihedral group, then a tighter bound holds; also, Reyes-Carocca and Speziali \cite{Pietro} studied Riemann surfaces of genus $g$ having an automorphism group of size $6g$ admitting a cyclic subgroup of index $2$. 

The present paper addresses the following two questions:
\begin{itemize}
\item does the upper bound $4g(\cX)+4$ hold for dihedral automorphism groups of curves of genus $g(\cX)\ge 2$ defined over any algebraically closed field?

\item for a given $g\ge 2$, what are the curves of genus $g$ with a large automorphism group $H$ admitting a cyclic subgroup of index $2$? Can they be described explicitly?

\end{itemize}

\textcolor{black}{To extend the upper bound $4g(\mathcal{X})+4$ for dihedral groups from the complex case to arbitrary algebraically closed fields, one could initially appeal to Grothendieck's tame lifting theorem \cite[Exposé XIII]{SGA1}, which guarantees the bound holds provided the characteristic does not divide the order of the group. 
In Section \ref{dint}, we present a general proof demonstrating that the dihedral bound remains valid over any algebraically closed field of characteristic different from $2$; see Theorem \ref{nd}.}


Concerning the second question, taking the sharp dihedral bound $4g+4$, which is attained in every even genus, as a natural threshold, we focus on the case where $|H|$ is strictly greater than $4g+4$; see \cite[Corollary~2.6]{BUJ2}.

In Section \ref{dint} we deal with the non-tame case. More precisely, Theorem \ref{nt} shows that if $p\mid |H|$ and $|H|>4g(\mathcal X)+4$, then $\mathcal X$ is birationally equivalent to an Artin--Schreier curve
\[
y^p-y=x^m,
\]
and that this is the only exceptional family.

In the rest of the paper we assume that $p\nmid |H|$. 
Section \ref{sez4} shows that, up to birational equivalence, the curve has equation
\[
\mathcal X(N,1,s):\qquad y^N=x(1+x)^s,
\]
and that the group $H$ is generated by a cyclic automorphism of order $N$ and an involution satisfying an explicit relation; see Theorem \ref{abgr}. If we put $M:=N-2g(\mathcal X)$, then $M\mid N$, and the arithmetic conditions relating $N$, $M$, and $s$ give strong restrictions on the possible genera and on the possible values of $|H|$. Proposition \ref{espliciti} gives explicit examples whenever the congruence condition $N\mid s(s+2)$ is satisfied.

A first consequence of this analysis is the description of the tame spectrum above $4g(\mathcal X)+4$. 
Theorem \ref{Spectotale} proves that the possible orders $2N>4g(\mathcal X)+4$ are exactly
\[
4g(\mathcal X)\frac{k}{k-1},
\qquad
2\le k\le g(\mathcal X),
\quad k-1\mid 2g(\mathcal X),
\quad \gcd(2g(\mathcal X),k)\le 2.
\]

\textcolor{black}{It should be remarked that in the context of Riemann surfaces, the problem of establishing whether a cyclic group $G$ of size $N$ of automorphisms of a compact Riemann surface can be extended to a group of size $2N$ has been thoroughly investigated, for instance, in \cite{SING} and \cite{BUJ}. However, these classical results typically establish existence without providing explicit equations for the curves involved; in contrast, our approach explicitly constructs these algebraic models. From Theorem 4.1 in \cite{BUJ} it is not difficult to deduce that the statement of Theorem \ref{Spectotale} holds for the complex case. Furthermore, while the extension to algebraically closed fields of any characteristic coprime to $2N$ could alternatively be deduced via Grothendieck's tame lifting theorem \cite[Exposé XIII]{SGA1}, our proof is entirely self-contained and independent of such machinery.}
%

As a Corollary to Theorem \ref{Spectotale}, the largest values are
\[
8g,\quad 6g,\quad \frac{16}{3}g,\quad 5g,\quad \frac{24}{5}g,\quad \frac{32}{7}g,\ldots
\]
according to the arithmetic of the genus.
Section \ref{classi} is devoted to the top part of this spectrum. We prove that for every genus $g\ge 2$ with $p\nmid g$ there is a unique hyperelliptic curve admitting such a group of order $8g$; see Theorem \ref{uniqueh}. In the non-hyperelliptic case, Proposition \ref{maindi} gives the sharp upper bound $|H|\le 6g$. We then classify explicitly the families attaining the values $6g$, $\frac{16}{3}g$, $5g$, and $\frac{24}{5}g$.

The paper is organized as follows. Section \ref{background} collects the background material used later. Section \ref{dint} deals with groups whose order is divisible by the characteristic and contains, in particular, the extension of the dihedral bound and the classification of the exceptional non-tame family. Section \ref{sez4} treats the tame case and develops the structural results. Finally, Section \ref{classi} describes the spectrum of the possible orders above $4g(\mathcal X)+4$ and classifies the curves occurring at the top of the spectrum.

\section{Preliminaries}\label{background}

In this section, we review the known results on automorphism groups that are pertinent to the rest of the paper. Our notation and terminology are standard; see \cite[Chapter 11]{HKT} for details.
Let $\mathcal{X}$ denote a projective, geometrically irreducible, non-singular algebraic curve defined over an algebraically closed field $\mathbb{K}$ with characteristic $p$, and let $\aut(\mathcal{X})$ be its automorphism group.

For a subgroup $G$ of $\aut(\cX)$, let $\bar \cX$ be a non-singular model of $\K(\cX)^G$, that is,
a projective non-singular geometrically irreducible algebraic
curve with function field $\K(\cX)^G$, where $\K(\cX)^G$ consists of all elements of $\K(\cX)$
fixed by every element in $G$. Usually, $\bar \cX$ is called the
quotient curve of $\cX$ by $G$ and denoted by $\cX/G$. The field extension $\K(\cX)|\K(\cX)^G$ is Galois of degree $|G|$, and let $\Phi$ be the associated rational map $\cX\to \bar{\cX}$.

A point $P \in \mathcal{X}$ is called a ramification point of $G$ if the stabilizer $G_P$ of $P$ in $G$ is non-trivial; the ramification index $e_P$ at $P$ is given by $|G_P|$. A point $\bar{Q}\in\bar{\cX}$ is a branch point of $G$ if there is a ramification point $P\in \cX$ such that $\Phi(P)=\bar{Q}$. 

\begin{thm}\cite[Theorem 11.49]{HKT} \label{stabilizer}
If the characteristic of the ground field does not divide the size of $G_P$, then $G_P$ is a cyclic subgroup of $G$.
\end{thm}

The $G$-orbit of $P \in \mathcal{X}$ is the subset of $\mathcal{X}$
$$o = \{ R \mid R = g(P), \, g \in G \},$$
and it is called {\em long} if $|o| = |G|$; otherwise, $o$ is {\em short}. For a point $\bar{Q}$, the $G$-orbit $o$ lying over $\bar{Q}$ consists of all points $P \in \mathcal{X}$ such that $\Phi(P) = \bar{Q}$. If $P\in o$ then $|o|=|G|/|G_P|$ and hence $\bar{Q}$ is a branch point if and only if $o$ is a short $G$-orbit. Note that it is possible for $G$ to have no short orbits, which occurs if and only if every non-trivial element of $G$ is fixed-point-free on $\mathcal{X}$, meaning the cover $\Phi$ is unramified. On the other hand, $\mathcal{X}$ has a finite number of short $G$-orbits.

A subgroup of $\aut(\mathcal{X})$ is a $p'$-group (or a prime-to-$p$ group) if its order is prime to $p$. A subgroup $G$ of $\aut(\mathcal{X})$ is {\em tame} if the stabilizer of any point in $G$ is a $p'$-group. Otherwise, $G$ is {\em non-tame} (or {\em wild}). While every $p'$-subgroup of $\aut(\mathcal{X})$ is tame, the converse is not always true.

\begin{thm}\cite[Theorem 11.60]{HKT} \label{theorem11.60HKT}
If $G$ is a subgroup of $\aut(\mathcal{X})$ that fixes a point and $p \nmid |G|$, then
$$|G| \leq 4g(\mathcal{X}) + 2.$$
\end{thm}

\begin{thm}\cite[Theorem 11.79]{HKT} \label{theorem11.79}
If $G$ is an abelian subgroup of $\aut(\mathcal{X})$, then
$$|G| \leq \begin{cases}
4g(\mathcal{X}) + 4, & \text{if } p \neq 2, \\
4g(\mathcal{X}) + 2, & \text{if } p = 2.
\end{cases}$$
\end{thm}

Provided that $G$ is a $p'$-group, the Hurwitz genus formula applied to $G$ is given by:
\begin{equation}
\label{Hurwitz}
2g(\mathcal{X}) - 2 = |G|(2g(\bar{\mathcal{X}}) - 2) + \sum_{P \in \mathcal{X}} d_P,
\end{equation}
where
\begin{equation}
\label{differente}
d_P = \sum_{i \geq 0} (|G_P^{(i)}| - 1).
\end{equation}

If a curve has genus $0$ it is called rational and its automorphism group is isomorphic to $\mathrm{PGL}(2, \mathbb{K})$.
The following result is a consequence of Dickson's classification of finite subgroups of $\mathrm{PGL}(2, \mathbb{K})$ (see \cite[Theorem 11.91]{HKT}).
\begin{prop}
\label{dickson}
Any non-trivial tame automorphism of a rational curve that fixes a point has exactly two fixed points.
\end{prop}

\begin{prop}\cite{DGT}\label{n=1}
Let $\mathcal{X}$ be an algebraic curve of genus $g\geq 2$ defined over a field of characteristic $p$. If $G$ is a cyclic automorphism group of $\mathcal{X}$ of order coprime with $p$ then $G$ cannot act with exactly one short orbit.
\end{prop}

\section{Groups of order divisible by $p$}\label{dint}

Throughout the paper,  $N$ is a positive integer 
and $H$ is a subgroup of the automorphism group of a curve $\cX$ of genus $g=g(\cX)$, admitting a cyclic subgroup $G$ of order $N$ and index $2$ in $H$. 
In the sequel, $\rho$ is a generator of $G$ and $\sigma$ is an element in $H\setminus G$. In this Section, we deal with the case where $N$ is a multiple of the characteristic $p\geq 3$.

The following statement can be deduced from the main result in \cite{DGT}.
\begin{prop}\label{MAINvecchio}
Let $p \geq 3$, and let $\mathcal{X}$ be a curve of genus $g(\mathcal{X}) \geq 2$ defined over a field of characteristic $p$ with a cyclic automorphism group $G$ of order $N > 2g(\mathcal{X}) + 2$ and divisible by $p$. Then, up to birational equivalence, $N = pm$ with $m > 1$ and $(p,m)=1$; also, 
  $g(\mathcal{X}) = \frac{(p-1)(m-1)}{2}$,
\begin{equation}\label{ntame}
\mathcal{X}: y^p - y = a(x^m - b)
\end{equation}
for some $a, b \in \mathbb{K}$, and $G=\langle \rho \rangle $ with
$$
\rho(x)=\lambda x, \quad \rho(y)=y+1.
$$
with $\lambda$ a primitive $m$-th root of unity.
\end{prop}

\begin{prop}\label{Equi}
Up to birational equivalence, $a=1$ and $b=0$ can be assumed in \ref{ntame}. 
\end{prop}
\begin{proof}
Let $\alpha$ and $\beta$ be elements in $\mathbb K$ such that $\alpha^m=a$ and $\beta^p-\beta=-ab$.
Also let $x=\frac{1}{\alpha} x'$ and $y=y'+\beta$. Then
$(y')^p+\beta^p-y'-\beta=a\frac{1}{\alpha^m}(x')^m-ab$, whence
$(y')^p-y'=(x')^m$.
\end{proof}

We first prove the following preliminary result.

\begin{thm}\label{29set} Let $\cX$ be a curve of genus $g(\cX)\ge 2$ defined over a field of positive characteristic $p\ne 2$, and $G$ be a cyclic subgroup of $\aut(\cX)$ of order $N>2g(\cX)+2$. Assume $p\mid N$ and that there exists an involution $\sigma\in  \aut(\cX)\setminus G$ that normalizes $G$. Then, up to birational equivalence, $\cX: y^p-y=x^m$, $G$ is as in Proposition \ref{MAINvecchio}, $m$ is odd, and 
$$
\sigma(x)=-x, \quad \sigma(y)=-y+i,
$$
with $i \in \mathbb{F}_p$. In particular, there are at most $p$ involutions in $\aut(\cX)\setminus G$ that normalize $G$.
\end{thm}
\begin{proof} 
As $p\ge 3$ divides the size of a cyclic group $G$ of automorphisms of size larger than $2\gg+2$, by Propositions \ref{MAINvecchio} and \ref{Equi} the only possibility is that $\cX$ is birationally equivalent to the curve with equation $y^p-y=x^m$ and that $G$ is as in the statement of Proposition \ref{MAINvecchio}.

The group $G$ fixes the common pole $P_\infty$ of $x$ and $y$ and acts with just another short orbit $\Delta$, consisting of the $p$ zeros of $x$.

As $\sigma$ normalizes $G$, it acts on the short orbits of $G$. As $|\Delta|=p>1$, we have that $\sigma$ preserves both the point $P_\infty$ and the set $\Delta$. Let $H=\langle G,\sigma \rangle$. Then the factor group $H/G$ acts on the quotient curve $\mathcal X/G$.  Then $\sigma G\in H/G$ fixes the points $\bar P_\infty$ and $\bar P_0$ in $\mathcal X/G$ lying under $P_\infty$ and $\Delta$, respectively. Note that the fixed field of $G$ is $\mathbb K(x^m)$; also, $\bar P_\infty$ is the only pole of $t=x^m$ in $\mathcal X/G$ and $\bar P_0$ is the only zero of  $t$ in $\mathcal X/G$. As there is just one involution of the projective line $\mathcal X/G$ fixing $\bar P_0$ and $\bar P_{\infty}$, we have $(\sigma G)(t)=-t$. This implies $\sigma(x^m)=-x^m$ and hence $(\sigma(x)/x)^m=-1$. Then $\sigma(x)=vx$ for some $v$ with $v^m=-1$. As $\sigma$ is an involution, we have $v^2=1$. This can only happen when $m$ is odd and $v=-1$. Also, from $\sigma(y)^p-\sigma(y)=\sigma(x)^m$ it follows that 
$$
0=\sigma(y)^p+y^p-\sigma(y)-y=(\sigma(y)+y)^p-(\sigma(y)+y),
$$
that is $\sigma(y)=-y+i$ for some $i\in \{0,1,\ldots,p-1\}$.
\end{proof}

As a corollary of Theorem \ref{29set}, we obtain that the bound $4g+4$ for dihedral automorphism groups holds also for orders divided by the characteristic $p\ge 3$.

\begin{thm}\label{nd} Let $\cX$ be a curve of genus $g(\cX)\ge 2$ defined over a field of positive characteristic $p\ne 2$. Let $D$ be a dihedral group of automorphisms of $\cX$.
Then 
$$
|D|\le 4g(\cX)+4.
$$
\end{thm}
\begin{proof} 
If the characteristic does not divide $|D|$, the result is a consequence of the Tame Lifting Theorem and the fact that the statement is valid for Riemann surfaces; see Introduction.
So, assume that $p\geq 3$ divides $|D|$.
If $|D|\le 4g(\cX)+4$, there is nothing to prove. Hence we may assume $|D|>4g(\cX)+4$. Then,
\[
N=\frac{|D|}{2}>2g(\cX)+2,
\]
and Proposition \ref{MAINvecchio} together with Theorem \ref{29set} imply that $N=pm$, $m>1$, and there are at most $p$ involutions in $\aut(\cX)$ that normalize $G$. Therefore, there are at most $p$ involutions in $D$. Thus $N\le p$, a contradiction.
\end{proof}

The aim of the final part of the section is to prove the following result. 
\begin{thm}\label{nt}
Let $\cX$ be a curve of genus $g(\cX)\geq 2$ defined over an algebraically closed field of characteristic $p\ge 3$. Then $\cX$ has an automorphism group $H$ of size divisible by $p$ and greater than $4g(\cX)+4$ admitting a cyclic subgroup of index $2$ if and only if $\cX$ is birationally equivalent to the curve with equation 
$$
y^p-y=x^m
$$
for some $m>1$ with $GCD(m,p)=1$. In this case $g(\cX)=\frac{(p-1)(m-1)}{2}$ and $|H|=2pm$. Also, $H$ is generated by $\rho$ and $\eta$ where
$$
\rho(x)=v^2 x, \quad \rho(y)=y+1
$$
and
$$
\eta(x)=v x, \quad \eta(y)=-y
$$
with $v\in \mathbb K$ a primitive $2m$-th root of unity.
\end{thm}

Then the {\em only if} part in Theorem \ref{nt} is a Corollary to Propositions \ref{MAINvecchio} and \ref{Equi}.
To complete the proof of Theorem \ref{nt} we need the following statement.
\begin{prop}
Let $\cX$ be the curve defined over an algebraically closed field of characteristic $p\ge 3$ by the following equation:
$$
\cX:y^p-y=x^m
$$
for some $m>1$ with $GCD(m,p)=1$. 
Let $\rho$ and $\eta$ be defined as follows: $$
\rho(x)=v^2 x, \quad \rho(y)=y+1
$$
 and
$$
\eta(x)=v x, \quad \eta(y)=-y
$$
with $v\in \mathbb K$ a primitive $2m$-th root of unity.
Then 
\begin{itemize}
\item[(a)] both $\rho$ and $\eta$ are automorphisms of $\cX$;
\item[(b)] the order of $\rho$ is equal to $pm$ and the order of $\eta$ is $2m$;
\item[(c)] Let $u$ be the integer between $1$ and $pm$ such that $u\equiv 1 \pmod m$ and $u \equiv p-1 \pmod p$; then $\eta\rho\eta^{-1}=\rho^u$.
\item[(d)] the group generated by $\rho$ and $\eta$ has size $2pm$.
\end{itemize}
\end{prop}
\begin{proof}
The fact that $\rho(y)^p-\rho(y)=\rho(x)^m$ depends on the fact that $(v^2)^m=1$; on the other hand, $v^m=-1$ implies that $\eta$ is an automorphism of $\cX$. This proves (a). For a positive integer $i$ we have 
$\rho^i(x)=v^{2i}x$, $\rho^i(y)=y+i$, $\eta^i(x)=v^ix$, $\eta^{i}(y)=(-1)^iy$. Since $v$ is a primitive $2m$-th root of unity and $GCD(p,m)=1$, we have that (b) holds. 

To prove (c), we compute $\eta \rho \eta^{-1}$. It turns out that $\eta \rho \eta^{-1}(x)=v^2 x$ whereas $\eta \rho \eta^{-1}(y)=y-1$. Now, let $u$ be the integer between $1$ and $pm$ such that $u\equiv 1 \pmod m$ and $u \equiv p-1 \pmod p$. Then $\rho^u(x)=v^2 x$ and $\rho^u(y)=y+p-1=y-1$, which implies $\eta\rho\eta^{-1}=\rho^u$.

Let $G=\langle \rho \rangle$ and $I=\langle \eta \rangle$. Note that $\eta^i\in G$ if and only if $i$ is even, that is, $|G\cap I|=m$. Then $|GI|=2pm$. To prove (d) we only need to note that by (c) we have $GI=IG$, that is, $\langle G,I \rangle $ coincides with $GI$.

\end{proof}

\section{Groups of order coprime to $p$}\label{sez4}



Throughout this section we assume
\begin{equation}\label{condizioneQD}
|H|>4\gg+4,
\end{equation} 
and, if the characteristic of the ground field is $p>0$, we assume that $p$ does not divide $2N$.
By Theorem \ref{nd}, $H$ is not a dihedral group.




Note that Condition \eqref{condizioneQD} implies that $\cX$ admits a cyclic group of automorphisms of size greater than $2g(\cX)+2$. From the classification obtained in \cite{DGT}, then $\cX$ must be isomorphic to a curve of equation
$$
\mathcal X(N,r,s):y^N=x^r(1+x)^s
$$
where $1\le r,s<N$, $GCD(N,r,s)=1$. Also, $G$ is the subgroup generated by
\begin{equation}\label{defro}
\rho(x)=x,\quad \rho(y)=\xi y,
\end{equation}
where $\xi$ is a primitive $N$-th root of unity.
We recall some facts on curves of type $\mathcal X(N,r,s)$ from \cite{DGT}.
\begin{thm}\label{recap}
Let $$
\mathcal X(N,r,s):y^N=x^r(1+x)^s
$$
where $1\le r,s<N$, $GCD(N,r,s)=1$, be defined over an algebraically closed field $\mathbb K$.
If the characteristic of $\mathbb K$ is $p>0$, assume that $p$ does not divide $N$. Then the following holds:
\begin{enumerate}
\item[(i)] The genus $g$ of $\cX(N,r,s)$ satisfies
$$
2g-2=N-GCD(N,r)-GCD(N,s)-GCD(N,r+s).
$$
\item[(ii)] There are exactly $M_1\coloneqq GCD(N,r)$ places of $\mathcal X(N,r,s)$, say $P_1,\ldots,P_{M_1}$, centered at the affine point $(0,0)$.

\item[(iii)] There are exactly $M_2\coloneqq GCD(N,s)$ places of $\mathcal X(N,r,s)$, say $R_1,\ldots,R_{M_2}$, centered at the affine point $(-1,0)$.

\item[(iv)] There are exactly $M_3\coloneqq GCD(N,r+s)$ places of $\mathcal X(N,r,s)$, say $Q_1,\ldots,Q_{M_3}$, centered at the only infinite point of $\mathcal X(N,r,s)$.

\item[(v)] the divisors of $x$, $x+1$, and $y$ are
$$
\div(x)=\frac{N}{M_1}(P_1+\ldots+P_{M_1})-\frac{N}{M_3}(Q_1+\ldots+Q_{M_3}),
$$

$$
\div(x+1)=\frac{N}{M_2}(R_1+\ldots+R_{M_2})-\frac{N}{M_3}(Q_1+\ldots+Q_{M_3}),
$$

and
$$
\div(y)=\frac{r}{M_1}(P_1+\ldots+P_{M_1})
+\frac{s}{M_2}(R_1+\ldots+R_{M_2})
-\frac{r+s}{M_3}(Q_1+\ldots+Q_{M_3}),
$$

\item[(vi)] The genus of the quotient curve 
$\mathcal X(N,r,s)/\langle \rho \rangle$ is $0$.

\item[(vii)] The only short orbits of the places of  $\mathcal X(N,r,s)$ under the action of $\langle \rho \rangle $ are $\Omega_1=\{P_1,\ldots,P_{M_1}\}$,
$\Omega_2=\{R_1,\ldots,R_{M_2}\}$, $\Omega_3=\{Q_1,\ldots,Q_{M_3}\}$.

\end{enumerate}
\end{thm}

Before delving into the investigation of the group $H$, we prove a series of results about the family of curves $\mathcal X(N,r,s)$.
\begin{lem}\label{isomorfismi}
Curves $\mathcal X(N,r,s)$, $\mathcal X(N,s,r)$, $\mathcal X(N,N-r-s,s)$, $\mathcal X(N,r,N-r-s)$ are birationally equivalent. 
\end{lem}
\begin{proof} Let $y^{N}=x^r(x+1)^s$. Then
\begin{itemize}
\item for $x=1/u$, $y=v/u$ we have $v^N=u^{N-r-s}(1+u)^s$;
\item for $x=-u-1$ and $y=\zeta v$ with $\zeta^{N}=(-1)^{r+s}$ we have $v^N=u^s(1+u)^r$;
\item for $x=-\frac{u}{u+1}$, $y=\zeta \frac{v}{u+1}$ with $\zeta^{N}=(-1)^{r}$ we have $v^N=u^r(u+1)^{N-r-s}$
\end{itemize}
\end{proof}

Let $\rho$ be as in \eqref{defro}.
As $G=\langle \rho \rangle$ is normal in $H$, the factor group $H/G$ has order $2$ and acts on the projective line. By the Riemann-Hurwitz genus formula, $H/G$ has exactly $2$ fixed points on the projective line. For any $\eta\in H\setminus G$, this  action is isomorphic to that of $\langle \eta \rangle$ on the orbits of $\cX=\cX(N,r,s)$ under $G$. As $G$ acts exactly with three short orbits, we deduce that one of them is fixed and the other two are swapped. Moreover, there is also a long orbit $\Delta$ under $G$ which is fixed by $H$. From now on, we denote with $\sigma$ a non-trivial element in the stabilizer of any point in $\Delta$. Then $\sigma$ is an involution not belonging to $G$.

\begin{prop}\label{3cases} Let $M_i\ge 1$, $i=1,2,3$, as in Theorem \ref{recap}. Then without loss of generality we may assume that $M_1=M_3=1$ and $M_2> 2$.
\end{prop}
\begin{proof} 
Note that (i) together with Condition \eqref{condizioneQD} implies that at least one integer among 
$M_1$, $M_2$, $M_3$ is greater than $1$. Without loss of generality, by Lemma \ref{isomorfismi}, we may assume that  $M_2>1$. We next show that 
$M_1=M_3=1$ and $M_2>2$.

From the above discussion, for a $\eta\in H\setminus G$ we can distinguish three cases:

\begin{itemize}
\item[(A)] $\eta(\Omega_3)=\Omega_3$.
Here $\Omega_1$ and $\Omega_2$ are swapped. If $\bar \eta$ is the non-trivial element in $H/G$, then we have
$$
\bar \eta(0)=-1,\bar \eta(-1)=0, \bar \eta(\infty)=\infty
$$
Therefore the action of $\eta$ on $\mathbb K(x)$ is determined; in particular
$$
\eta(x)=-(x+1)
$$

\item[(B)] $\eta(\Omega_2)=\Omega_2$
Here $\Omega_1$ and $\Omega_3$ are swapped. Arguing as in case (A) we have
$$
\eta(x)=1/x
$$

\item[(C)] $\eta(\Omega_1)=\Omega_1$
Here $\Omega_2$ and $\Omega_3$ are swapped.
Arguing as in case (A) we have
$$
\eta(x)=-\frac{x}{x+1}
$$
\end{itemize}
In order to investigate the possible values of $\eta(y)$,  note that
$$
\eta(y)^N=\eta(x)^r(1+\eta(x))^s.
$$
Therefore,
$$
\left(\frac{\eta(y)}{y}\right)^N=\frac{\eta(x)^r(1+\eta(x))^s}{x^r(1+x)^s}.
$$
The computation of the divisor  of $(\frac{\eta(y)}{y})$ will provide the statement will provide a contradiction in Cases (A) and (C).
\begin{itemize}

\item if (A) holds, then
$$
\div\left(\left(\frac{\eta(y)}{y}\right)^N\right)= \frac{N(s-r)}{M_1}(P_1+\ldots+P_{M_1})-\frac{N(s-r)}{M_2}(R_1+\ldots+R_{M_2})
$$
Then $\frac{N(s-r)}{M_1}$ must be an integer multiple of $N$ and $M_1$ is a divisor of $s-r$. As $M_1$ divides $r$ but $GCD(M_1,s)=1$, the only possibility is  $M_1=1$. Then $|\Omega_1|=1$ and hence $|\Omega_2|=M_2=1$, a contradiction.

\item if (B) holds, then
$$
\div\left(\left(\frac{\eta(y)}{y}\right)^N\right)= \frac{N(2r+s)}{M_3}(Q_1+\ldots+Q_{M_3})-\frac{N(2r+s)}{M_1}(P_1+\ldots+P_{M_1})
$$
Then $\frac{N(2r+s)}{M_1}$ must be an integer multiple of $N$ and $M_1$ is a divisor of $2r+s$. As $M_1$ divides $r$ but $GCD(M_1,s)=1$, the only possibility is  $M_1=1$. Then $|\Omega_1|=1$ and hence $|\Omega_3|=M_3=1$ and
$$
\div\left(\frac{\eta(y)}{y}\right)=(2r+s)(Q_1-P_1).
$$
Also, $2g-2=N-2-M_2$ together with Condition \eqref{condizioneQD} gives $M_2>2$.

\item if (C) holds, then
$$
\div\left(\left(\frac{\eta(y)}{y}\right)^N\right)= \frac{N(r+2s)}{M_3}(Q_1+\ldots+Q_{M_3})-\frac{N(r+2s)}{M_2}(R_1+\ldots+R_{M_2})
$$
Arguing as in case (B) we have $M_2=M_3=1$, a contradiction.
\end{itemize}
\end{proof}

From now on, we let $M=M_2$ as in Proposition \ref{3cases} and $\Omega$ will denote the short orbit of $G$ of size $M$. 

We now show that $r=1$ can be assumed in $\mathcal X(N,r,s)$.

\begin{lem} The curve $\mathcal X(N,r,s)$ is birationally equivalent to $\mathcal X(N,1,w)$ for some $w$. Also, $GCD(N,s)=1$ if and only if $GCD(N,w)=1$. 
\end{lem}
\begin{proof} As $M_1=1$ there exists an integer $h$ such that
$rh=1+N\ell$ for some $\ell\in \mathbb Z$. Let $w\in \{0,\ldots, N-1\}$ be such that $sh\equiv w \pmod N$, that is, $sh=w+Nz$.
Consider the following rational functions on $\mathcal X(N,r,s)$:
$u=x$ and $v=\frac{y^h}{x^\ell(1+x)^z}$. Then
$$
v^N=\frac{y^{hN}}{x^{\ell N}(1+x)^{zN}}=
\frac{x^{rh}(1+x)^{sh}}{x^{\ell N}(1+x)^{zN}}=x(1+x)^{w}
$$
\end{proof}

\begin{cor} Up to birational equivalence, $r=1$ can be assumed.
\end{cor}

Two classes of curves $\mathcal X(N,r,s)$ for which Condition \eqref{condizioneQD} holds  are described as follows.

\begin{exmp}\label{gHERM}
Let $m>3$ be an integer such that $p\nmid m^2-1$, and
\begin{equation}\label{HGEN}
\mathcal H(m): y^{m^2-1}=x(x+1)^{m-1}
\end{equation}
be defined over an algebraically closed field $\mathbb K$ of characteristic $p$.
For a primitive $(m^2-1)$-th root of unity $\xi\in \mathbb K$, let $\rho$ and $\sigma$ be the automorphisms of $\mathcal H(m)$ defined as follows: 
\begin{itemize}
\item
$\rho(x)=x$, $\rho(y)=\xi y$
\item $\sigma (x)=\frac{1}{x}$, $\sigma(y)=\frac{x+1}{y^m}$
\end{itemize}
It is straightforward to check that
$$
\rho^{m^2-1}=1, \quad \sigma^2=1, \quad \sigma \rho \sigma=\rho^{m^2-m-1}.
$$
Then, for $N=m^2-1$, the group $H$ generated by $\rho$ and $\sigma$ has order $2N$. The genus of $\cX$ is $g=\frac{|H|+2-2m}{4}$, whence $|H|>4g+4$ as $m>3$.
\end{exmp}

\begin{exmp}\label{gH2}
Let $m>2$ be an integer such that $p\nmid m^2-1$, and
\begin{equation}\label{gH1}
\mathcal L(m):y^{m^2-1}=x(x+1)^{m}
\end{equation}
be defined over an algebraically closed field $\mathbb K$ of characteristic $p$.

Let $\rho$ be defined as in Example \ref{gHERM}.
Then $\sigma(x)=-(x+1), \sigma(y)=-y^{m}/(x+1)$ is an  automorphism which normalizes $\langle \rho \rangle$, since
 $\sigma\rho\sigma^{-1}=\rho^m$. Also $\sigma^2\in G$, so the group generated by $G$ and $\sigma$ has order $2N=2(m^2-1)$. Note that the order of $\sigma$ is $2$ if $m$ is even, $4$ if $m$ is odd. 
The genus of $\cX$ is $g=\frac{|H|-2(m+1)}{4}$, , whence $|H|>4g+4$ as $m>2$.

Note that by Lemma \ref{isomorfismi}, Example \ref{gH2} is isomorphic to 
$$
y^{m^2-1}=x(x+1)^{m^2-m-2}
$$
\end{exmp}

In the following statements we determine some properties of the group $H$.

\begin{lem}\label{24stab}
Let $L$ be the (unique cyclic) subgroup of $G$ of order $N/M$. Then $\cX/L$ is a rational curve and the factor group $H/L$ is the dihedral group of order $2M$.
\end{lem}
\begin{proof}
First, observe that $L$ is a characteristic subgroup of $G$, hence it is normal in $H$.
Moreover, $L$ fixes the two points fixed by $G$ and also each point in the orbit $\Omega$ of size $M$ under the action of $G$.
By the Riemann-Hurwitz genus formula, together with (i) of Theorem \ref{recap},
$$
N-M-2=2g(\cX)-2=|L|(2g(\cX/L)-2)+2(|L|-1)+M(|L|-1),
$$
which implies that $\cX/L$ is rational.
The factor group $H/L$ acts on the projective line as follows:
\begin{itemize}
\item has a short orbit of order $2$;
\item has two short orbits of order $M$, namely the images of $\Omega$ and $\Delta$ under the projection $\mathcal X\to \mathcal X/L$;
\item every point in these two short orbits is fixed by some involution in $H/L$.
\end{itemize}
Then the claim follows from \cite[Theorem 11.91]{HKT}.
\end{proof}

\begin{cor}\label{semidiretto} If $N/M$ is coprime to $2M$, then $H$ is a semidirect product of a cyclic group of order $N/M$ and a dihedral group of order $2M$.
\end{cor}
\begin{proof} The claim is a consequence of Lemma \ref{24stab} and the Schur-Zassenhaus Theorem (see e.g. \cite[A.5(XI)]{HKT}).
\end{proof}

\begin{lem}\label{qadr} Let $L$ be as in Lemma \ref{24stab} and let $\rho$ be a generator of $G$. Then for any $\eta \in H\setminus G$ there exists $l\in \{0,\ldots, |L|-1\}$ such that
$
\eta^2=\rho^{lM}.
$
    \end{lem} 
\begin{proof}
By Lemma \ref{24stab} $H/L$ is dihedral and as $\eta L$ does not belong to the cyclic subgroup of $H/L$ of order $M$, we have 
$(\eta L)^2=L$. Therefore $\eta^2$ belongs to $L$ and the claim follows.
\end{proof}

As a consequence, we have the following result. 

\begin{thm}\label{abgr} 
The group $H$ has the following description as an abstract group:
$$
H=\langle \rho, \sigma \mid \rho^N=\sigma^2=1,\sigma \rho \sigma= \rho^{lM-1}\rangle,
$$
for some $l\in 
\{1,\ldots, \frac{N}{M}-1\}$ such that
\begin{itemize}
    \item  $GCD(N,lM-1)=1$,
    \item  $GCD(N,lM)=M$, and
    \item  $lM\equiv 2 \pmod{N/M}$.
\end{itemize}
In particular, either $GCD(M,N/M)=1$ or  $GCD(M,N/M)=2$.
\end{thm}
\begin{proof} The description of $H$ follows from Lemma \ref{qadr} applied to $\eta=\sigma\rho$, taking into account that for $l=0$ the group $H$ would be a dihedral group. As $\sigma \rho \sigma$ must be a generator of $\langle \rho \rangle$, condition $GCD(N,lM-1)=1$ holds. 
To prove that $GCD(N,lM)=M$, assume by contradiction that $GCD(N,lM)=vM$ for some $v>1$. Let $V$ be the subgroup of $H$ generated by $\rho^{vM}$. As $V$ is a characteristic subgroup of $G$, we have that $V$ is normal in $H$. The factor group $H/V$ has order $2vM$. We show that $H/V$ is a dihedral group. The element $\bar \rho=\rho V$ has order $vM$ and $\bar \sigma=\sigma V$
is an involution. Also,
$$
\bar \sigma \bar \rho \bar \sigma \bar \rho =\sigma\rho\sigma\rho V=\rho^{lM}V=V.
$$

We compute the genus $\bar g$ of the quotient curve $\cX/V$. By the Riemann-Hurwitz genus formula, taking into account that $V$ fixes $M+2$ points and acts semiregularly on the other points of $\cX$, we have
$$
N-M-2=\frac{N}{vM}(2\bar g-2)+(M+2)\left(\frac{N}{vM}-1\right);
$$
whence,
$
vM=2\bar g +M
$
and $\bar g=(v-1)M/2$.
As $M>2$, we have $\bar g\ge 2$. Also, $H/V$ is a dihedral group of size $2vM>4\bar g+4=2(v-1)M+4$ acting on a curve of genus $\bar g=(v-1)M/2$, a contradiction to Theorem \ref{nd}.

In order to prove $lM-2\equiv 0 \pmod{N/M}$, observe that
$$
\rho=\sigma (\sigma \rho \sigma) \sigma=\rho^{(lM-1)^2},
$$
which yields that $N$ divides 
$(lM-1)^2-1=lM(lM-2)$. As $GCD(N/M,l)=1$, this is equivalent to $N/M$ dividing $lM-2$. Finally,  by $lM\equiv 2 \pmod{N/M}$, $GCD(M,N/M)$ must divide $2$, which completes the proof. 
\end{proof}

Theorem \ref{abgr} has the following series of consequences.

\begin{thm}\label{tame-normal-form}
Let $\cX$ be a curve of genus $g(\cX)\geq 2$ with an automorphism group of size greater than $4g(\cX)+4$ and coprime with $p$, admitting a cyclic subgroup of index $2$. Then, up to birational equivalence, $\cX$ is defined by
\begin{equation}\label{esplicita}
y^{k\frac{2g(\cX)}{k-1}}=x(1+x)^{v\frac{2g(\cX)}{k-1}},
\end{equation}
where $k$ and $v$ satisfy $1<k<g(\cX)+1$, $k-1\mid 2g(\cX)$, $v\in \{1,\ldots, k-1\}$, $GCD(2g(\cX),k)\leq 2$, and $GCD(k,v)=1$. 
\end{thm}
\begin{proof}
By Proposition \ref{3cases}, the curve is isomorphic to a curve $\cX(N,1,s)$ with $M=GCD(N,s)>2$ and $GCD(N,s+1)=1$; also, $2g(\cX)=N-M$. Let $k=\frac{N}{M}$ and $v=\frac{s}{M}$. Then $k\geq 2$, $v\in \{1,\ldots,k-1\}$, and by $GCD(N,s)=M$ we have $GCD(k,v)=1$.
Also, $2g(\cX)=(k-1)M$, shows that $k-1$ is a divisor of $2g(\cX)$, whereas $M>2$ yields $k<g(\cX)+1$. Moreover, by $M=2g(\cX)/(k-1)$ we have that \eqref{esplicita} is an equation defining $\cX(N,1,s)$.
Finally, as $N/M=k$, Theorem \ref{abgr} yields $
GCD(k,2g/(k-1))\leq 2$. The claim follows from the fact that $GCD(k,2g/(k-1))=GCD(k,2g)$.
\end{proof}

\begin{cor}
If $N/M$ is odd, then $H$ is a semidirect product of a cyclic group of order $N/M$ and a dihedral group of order $2M$.
\end{cor}
\begin{proof}
 By Theorem \ref{abgr} and the assumption that $N/M$ is odd, we have $GCD(N/M,2M)=1$. The claim follows from Corollary \ref{semidiretto}. 
\end{proof}

\begin{thm}\label{14settembre}   Let $\cX$ be a curve with an automorphism group of size greater than $4g(\cX)+4$ admitting a cyclic subgroup of index $2$. Then, up to birational equivalence, $\cX=\cX(N,1,s)$ with
 $M=GCD(N,s)>2$ and $GCD(N,s+1)=1$, and there exists an integer $t$ such that $GCD(N,t)=M$, $GCD(N,t+1)=1$ and $t^2+2t\equiv 0 \pmod N$.
\end{thm}
\begin{proof}
Let $l$ be as in Theorem \ref{abgr} and  $t:=N-lM$. The claim follows from Theorem \ref{abgr}.
\end{proof}


The following result extends Examples \ref{gHERM} and \ref{gH2}, and provides a sufficient condition. 

\begin{prop}\label{espliciti} 
Let $N$ and $s$ be two natural numbers such that $p\nmid 2N$, $\gcd(N,s)>2$, and $\gcd(N,s+1)=1$. If $N$ divides $s^2+2s$, then $\cX(N,1,s)$ is a curve of genus
$g=\frac{N-\gcd(N,s)}{2}$ with an automorphism group $H$ of order $2N>4g+4$ which contains a cyclic subgroup of index $2$.
\end{prop}
\begin{proof}
Let $a=(s^2+2s)/N$ and let $\xi$ be a primitive $N$-th root of unity in $\mathbb K$. Let
$\rho$ and $\sigma$ be  defined as follows: 
\begin{itemize}
\item
$\rho(x)=x$, $\rho(y)=\xi y$;
\item $\sigma (x)=\frac{1}{x}$, $\sigma(y)=\frac{(x+1)^a}{y^{s+1}}$.
\end{itemize}
Clearly $\rho$ is an automorphism of $\cX(N,1,s)$ of order $N$. 
It is straightforward to check that $\sigma$ is an involution preserving $\cX(N,1,s)$. Indeed, $\sigma$ is an involution as $\sigma^2(x)=x$ and
$$
\sigma^2(y)=\sigma\left( \frac{(x+1)^a}{y^{s+1}}\right)=\left(\frac{(x+1)^a}{x^a}\right)\frac{y^{(s+1)^2}}{(x+1)^{a(s+1)}}=\frac{y^{(s+1)^2}}{y^{aN}}=y,
$$
it is an automorphism of $\cX(N,1,s)$ by
$$
\sigma(y)^N=\frac{(x+1)^{aN}}{y^{(s+1)N}}=\frac{(x+1)^{s^2+2s}}{x^{s+1}(x+1)^{s^2+s}},\qquad \sigma(x)(\sigma(x)+1)^s=\frac{1}{x}\left(\frac{1}{x}+1\right)^s=\frac{1}{x^{s+1}}(x+1)^s.
$$
Finally, it is straightforward to check that 
$$
\sigma \rho \sigma = \rho^{N-s-1}.
$$
\end{proof}

Combining Theorem \ref{14settembre} and Proposition \ref{espliciti}, we obtain the following description of the genera that occur for curves with an automorphism group of order greater than $4g+4$ and prime to $p$ admitting a cyclic subgroup of index $2$.

\begin{cor} For an integer $N> 6$ the genera $g\ge 2$ of curves with an automorphism group of size $2N>4g+4$ coprime with $p$ admitting a cyclic subgroup of index $2$ are precisely
the integers $\frac{N-M}{2}$ where $M$ is a divisor of $N$ such that $\gcd(M,N/M)\leq 2$, and such that there exists $v \in \left\{1,\ldots,\frac{N}{M}-1\right\}$ for which $\gcd(N,vM)=M$, $\gcd(N,vM+1)=1$ and $vM+2 \equiv 0 \pmod {N/M}$.
An example of a curve with that genus is
$
y^N=x(1+x)^{vM}.
$
In this case, the automorphisms in $H$ are explicitly described in Theorem \ref{espliciti}.
\end{cor}

In order to obtain further necessary conditions on the curves we are investigating, we can look at Weierstrass semigroups at certain points of $\cX$.

\begin{lem}\label{trenta} 
If $\cX=\cX(N,1,s)$, with $\gcd(N,s+1)=1$, then
$$
\div\left(\frac{\sigma(y)}{y}\right)=(s+2)(Q_1-P_1).
$$
Also, if $d=\gcd(N,s+2)$, then $d$ belongs to the Weierstrass semigroups at both $P_1$ and $Q_1$.

\end{lem}
\begin{proof}
The first part of the claim is a consequence of the proof of Proposition \ref{3cases}.
Moreover, let $\alpha, \beta$ be integers such that
$$
d=\alpha N-\beta (s+2).
$$
Recall that $\div\left(x\right)=N(P_1-Q_1)$.
Then
$$
\div \left( \left(x\right)^\alpha \left(\frac{\sigma(y)}{y}\right)^\beta\right)=d(P_1-Q_1),
$$
which proves the second part of the claim.
\end{proof}

By Lemma \ref{trenta} it is worth  investigating the Weierstrass semigroup at points $P_1,Q_1$.
Let
$$
A_N=\{(r,s) \in \mathbb Z^2 : r,s \ge 1 \text{ and }r+s\le N-1\}
$$
and
$$
A(r,s)=\{a \in \{1,\ldots, N-1\} \mid (ar \pmod N,as \pmod N)\in A_N\}.
$$

The following result can be deduced from \cite{sasaki}.

\begin{prop} Let $(r,s)\in A_N$ such that $\gcd(N,r)=1$. The map
$$\Psi: A(r,s)\longrightarrow Gap(P_1),$$
defined by $\Psi(a)=ar\pmod N$ is bijective.
\end{prop}
In \cite{sasaki} it is also implicitly proven the following.
\begin{prop} Let $(r,s)\in A_N$ such that $\gcd(N,s+1)=1$. The map
$$\Psi: A(r,s)\longrightarrow Gap(Q_1),$$
defined by $\Psi(a)=-a(s+r)\pmod N$  is bijective.
\end{prop}

For $r=1$, the elements in $A(1,s)$ are 
$a\in \{1,\ldots,N-1\}$ such that 
\begin{equation}\label{weis}
    N-1-a\ge as \pmod N \ge 1. 
\end{equation}
 
Therefore,  gaps at $P_1$ are $a\in \{1,\ldots,N-1\}$ such that Condition \eqref{weis} holds.
 Similarly,
gaps at $Q_1$ are $-a(s+1) \pmod N$
 for each $a\in \{1,\ldots,N-1\}$ such that Condition \eqref{weis} holds.

 Lemma \ref{trenta} together with the above mentioned results on the structure of Weierstrass semigroups at $P_1$ and $Q_1$, yield the following necessary condition.

 \begin{thm}\label{trentin} A necessary condition for a curve $\cX$ to have a tame automorphism group $H$ admitting a cyclic subgroup of index $2$  with
 $|H|>
  4g(\cX)+4$ is that 
 $\cX: y^N=x(1+x)^s$ with $\gcd(N,s+1)=1$, $d=\gcd(N,s+2)\ge 2$ and $N$ divides $ds$ or $N-1-d< ds \pmod N$.
  \end{thm}
Examples \ref{gH2} and \ref{gHERM} show that the necessary condition in Theorem \ref{trentin} is realized by natural infinite families.

\begin{rem}
Since there is an automorphism of the curve mapping $P_1$ to $Q_1$, the Weierstrass semigroups at $P_1$ and $Q_1$ coincide. Therefore,
the sets
$\{a\mid a \in A(1,s)\}$ and
$\{-a(s+1) \pmod N \mid a \in A(1,s)\}$ must be equal.
\end{rem}

\section{Classification results}\label{classi}

In this section we classify curves defined over an algebraically closed field $\mathbb K$ of positive characteristic $p$ which, for a fixed genus $g$, admit the largest possible automorphism group of size coprime with $p$ and with a cyclic subgroup of index $2$, distinguishing the hyperelliptic and the non-hyperelliptic case.

From now on, let $\mathcal S$ be the spectrum of possible sizes  $2N>4g+4$ of automorphism group of curves of genus $g\geq 2$ admitting a cyclic subgroup of index $2$, under the assumption that $p\nmid N$.

As a corollary of Theorem \ref{tame-normal-form}, we have that
\begin{equation}\label{spettro1}
    \mathcal{S}\subseteq \left\{4g\frac{k}{k-1}\,:\, 2\leq k\leq g, \, k-1 \text{ is a divisor of } 2g \text{ and } \gcd(2g,k)\leq 2 \right\}.
\end{equation}
In order to prove the equality in \eqref{spettro1}, we first provide the following consequence of 
Theorem \ref{tame-normal-form} and Proposition \ref{espliciti}.

\begin{thm} \label{suffspettro}
Let $k$ and $v$ be as in Theorem \ref{tame-normal-form}, and $\cX$ be the curve of genus $g$ defined by \eqref{esplicita}.
Assume that one of the following cases holds
\begin{itemize}
    \item $k$ is odd and $v$ is the inverse of $g$ modulo $k$, or
    \item $k$ is even and $v$ is the inverse of $g$ modulo $k/2$.
\end{itemize}
Then $\cX$ has an automorphism group of size greater than $4g+4$ and coprime with $p$, admitting a cyclic subgroup of index $2$. In particular, $4g\frac{k}{k-1}\in \mathcal{S}.
$
\end{thm}
\begin{proof}
    By Proposition \ref{espliciti}, if $N=k\frac{2g}{k-1}$ divides $$s(s+2)=v\frac{2g}{k-1}\left(v\frac{2g}{k-1}+2\right),$$ then the claim holds. This is equivalent to 
    $$
v\frac{2g}{k-1}+2\equiv 0 \pmod {k},
    $$
that is
$$
2vg\equiv 2 \pmod k,
$$
whence either 
$$
vg\equiv 1 \pmod k,
$$
or 
$$
vg\equiv 1 \pmod{k/2},
$$
according to $k$ odd or even, respectively.
\end{proof}

\begin{thm}\label{Spectotale}
The spectrum $\mathcal S$ of possible sizes $2N>4g+4$, with $p\nmid N$, of automorphism group of curves of genus $g\geq 2$ admitting a cyclic subgroup of index $2$ is
$$
\mathcal{S}= \left\{4g\frac{k}{k-1}\,:\, 2\leq k\leq g, \, k-1 \text{ is a divisor of } 2g \text{ and } \gcd(2g,k)\leq 2 \right\}.
$$
\end{thm}
\begin{proof}
For $g\geq 2$, let $2\leq k\leq g$ such that $k-1$ is a divisor of $2g$ and $\gcd(2g,k)\leq 2$. Our aim is to construct a curve of genus $g$ with an automorphism group of order $4g\frac{k}{k-1}$ with the required properties. Assume first that $k$ is odd. Then $\gcd(2g,k)=1$ and in particular $\gcd(g,k)=1$. So, let $v\in \{1,\ldots,k-1\}$ be the inverse of $g$ modulo $k$. Then clearly $\gcd(v,k)=1$ and by Theorem \ref{suffspettro} the curve defined by \eqref{esplicita} proves that $4g\frac{k}{k-1}\in \mathcal{S}$. 

Assume now that $k$ is even. If $k=2$, the curve $\mathcal{I}(g)$ in Theorem \ref{uniqueh} provides an example. So, let $k\geq 4$ even. Then $\gcd(2g,k)=2$ and hence $\gcd(g,k/2)=1$. Let $v\in \{1,\ldots,k/2-1\}$ be such that $vg\equiv 1\pmod{k/2}$. Clearly $\gcd(v,k/2)=1$ and hence $\gcd(v,k)\leq 2$. If $\gcd(v,k)=1$ then by Theorem \ref{suffspettro} the curve defined by \eqref{esplicita} proves that $4g\frac{k}{k-1}\in \mathcal{S}$. So, assume $\gcd(v,k)=2$. Then $v$ is even and $\gcd(v,k/2)=1$ implies that $k/2$ is odd. Let $v'=v+k/2$. Then $\gcd(v',k)=1$ and $v'g\equiv 1\pmod{k/2}$. As before, Theorem \ref{suffspettro} provides $4g\frac{k}{k-1}\in \mathcal{S}$. 
\end{proof}

Our aim is now to classify curves with an automorphism group attaining the higher part of the spectrum.
We start by dealing with the case $k=2$, corresponding to $N=4g$.

\begin{thm}\label{uniqueh} For every genus $g\ge 2$ such that $p\nmid g$ there exists a unique curve of genus $g$ with an automorphism group of order $8g$ admitting a cyclic subgroup of index $2$. Such a curve is
$$
\mathcal I(g): y^{4g}=x(1+x)^{2g}.
$$
\end{thm}
\begin{proof} 
By Theorem \ref{tame-normal-form}, if $N=4g$ the only possibility is $k=2$ and $v=1$, and an equation of the curve is that of
$\mathcal I(g)$. Moreover, by Theorem \ref{suffspettro},  $\mathcal I(g)$ has an automorphism group of order $8g$ admitting a cyclic subgroup of index $2$.
\end{proof}

Note that $\mathcal I(g)$ is a hyperelliptic curve, since by Lemma \ref{24stab} the involution $\rho^{2g}$ is such that the quotient curve $\mathcal I(g)/\langle \rho^{2g}\rangle$ is rational.
We prove that this is the only hyperelliptic curve among the curves defined by \eqref{esplicita}.

\begin{prop}\label{hyperelliptic-characterization}
The only hyperelliptic curve occurring in the (tame) spectrum
above $4g(\mathcal X)+4$ is the curve $\mathcal I(g)$ of
Theorem~\ref{uniqueh}, corresponding to $|H|=8g$.
\end{prop}

\begin{proof}
We already pointed out that $\mathcal I(g)$ is hyperelliptic.

Conversely, assume that $\mathcal X$ defined by \eqref{esplicita} is hyperelliptic, and let $\iota$
be its hyperelliptic involution. Since $\iota$ is central in
$\aut(\mathcal X)$, the group $G$ induces a cyclic group $\overline G$
on the rational curve $\mathcal X/\langle\iota\rangle$. Let $B$ be the
branch locus of the hyperelliptic covering, so that $|B|=2g+2$.

We first show that $\iota\in G$. Indeed, if $\iota\not\in G$, then $|\overline G|=N$. Since
$N>2g+2$, every point of $B$ would have to be fixed by $\overline G$,
as every non-fixed orbit of a tame cyclic subgroup of
$PGL(2,\mathbb K)$ has length $N$. This is impossible, since
$\overline G$ has only two fixed points and $|B|=2g+2>2$.

Therefore $\iota\in G$ and $|\overline G|=N/2$. Hence $B$ is the union
of $a\le 2$ fixed points of $\overline G$ and some orbits of length
$N/2$. Since $2g+2<N$, there is exactly one such orbit, and
\[
2g+2=a+\frac{N}{2}.
\]
Using $2g=N-M$, we obtain
\[
a=\frac{N}{2}-M+2.
\]
Since $a\le2$, it follows that $N/2\le M$, that is, $k=2$.
\end{proof}

As a by-product of Lemma \ref{24stab}, the following result is obtained.

\begin{prop}\label{iperl}\label{maindi}
If $\cX$ is not hyperelliptic, then  $|H|\le 6g$. Moreover, if $d=\gcd(N,s+2)$, then $d\ge 3$.
\end{prop}
\begin{proof}
As $\cX$ is not hyperelliptic, by Theorems \ref{uniqueh} and \ref{tame-normal-form} we have $k\geq 3$, whence $|H|\leq 6g$ follows.
The claim on $d$ follows from Lemma \ref{trenta} as $d$ is a non-gap at some point of $\cX$, and $\cX$ is not hyperelliptic.
\end{proof}



We are now going to discuss the case $k=3$ in order to classify curves attaining the upper bound in Proposition \ref{iperl}.
By Proposition~\ref{hyperelliptic-characterization}, all the curves
considered in the remainder of the section are non-hyperelliptic.
Therefore Proposition~\ref{maindi} applies, and
\(
d=\gcd(N,s+2)\geq 3.
\)

\begin{thm}\label{bello}
Let $\cX$ be a curve of genus $g\geq 2$ admitting an automorphism group $H$ of size $6g$
containing a cyclic group of order $3g$, with $p\nmid 3g$. Then $g\not\equiv 0 \pmod{3}$ and $\cX$ is birationally equivalent to 
one of the two curves
\begin{itemize}
\item $
\mathcal X(g):y^{3g}=x(x+1)^g$, if $g\equiv 1\pmod 3$;
\item $
\mathcal Y(g):y^{3g}=x(x+1)^{2g}
$,if $g\equiv 2\pmod 3$.
\end{itemize}
\end{thm}
\begin{proof}
By Theorem \ref{tame-normal-form}, $\cX$ is birationally equivalent to one of the curves with equation \eqref{esplicita}, where $k=3$ and $v\in \{1,2\}$, that is the curves  $\mathcal X(g)$ and $\mathcal Y(g)$. Moreover, $\gcd(2g,k)\leq 2$ yields that $3$ does not divide $g$. Assume first $v=1$, that is $s=g$. Then $\gcd(N,s+2)\geq 3$ yields $\gcd(g+2,6)\geq 3$, that is $g\equiv 1\pmod 3$. Similarly, if $v=2$ then $s=2g$ and $\gcd(N,s+2)\geq 3$ yields $g\equiv 2\pmod 3$.
Finally, observe that both $\mathcal X(g)$ and $\mathcal Y(g)$  admit an automorphism group of order $6g$ containing a cyclic subgroup of index $2$ by Theorem \ref{suffspettro}.
\end{proof}

\begin{rem} In the complex case, Nakagawa \cite{Nakagawa} classified curves with an automorphism of order $3g$. More recently, Reyes-Carocca and Speziali \cite{Pietro} proved that the full automorphism group of $\mathcal X(g)$ and $\mathcal Y(g)$, viewed as Riemann surfaces, is the direct product of a cyclic group of order $3$ and a dihedral group of order $2g$. Our approach is different and provides explicit affine equations in arbitrary characteristic prime to $3g$. 
\end{rem}

\begin{thm}
If $|H|<6g$, then $|H|\le \frac{16}{3}g$. Also, $|H|=\frac{16}{3}g$ if and only if $g\geq 9$ is odd, $g\equiv 0 \pmod{3}$, and $\cX$ is birationally equivalent to one of the two curves
\begin{itemize}
\item $\mathcal A(g):y^{\frac{8}{3}g}=x(x+1)^{\frac{2}{3}g}$,

\item $\mathcal B(g):y^{\frac{8}{3}g}=x(x+1)^{2g}$. 
\end{itemize}
\end{thm}

\begin{proof}
The case $|H|=16/3g$ corresponds to $k=4$ in Theorem \ref{tame-normal-form}. Therefore, $g\geq 4$ and the possibilities for $v$ in Equation \ref{esplicita} are $v=1$ and $v=3$, which correspond to the curves $\mathcal A(g)$ and $\mathcal B(g)$. Moreover, as $k-1=3$ is a divisor of $2g$, $g\equiv 0\pmod 3$, while $\gcd(k,2g)\leq 2$ yields that $g$ is odd.
To conclude the proof, we can observe that both $\mathcal A(g)$ and $\mathcal B(g)$ have an automorphism group of order $\frac{16}{3}g$ admitting a cyclic group of index $2$ by Theorem \ref{suffspettro}.
\end{proof}

\begin{thm}
If $|H|<16/3g$, then $|H|\le 5g$. Also, $|H|=5g$ if and only if $g\geq 6$ is even, $g\not\equiv 0 \pmod 5$, and $\cX$ is birationally equivalent to one of the curves
\begin{itemize}
\item $\mathcal C(g):y^{\frac{5}{2}g}=x(x+1)^{\frac{g}{2}}$, if $g \equiv 1 \pmod{5}$

\item $\mathcal D(g):y^{\frac{5}{2}g}=x(x+1)^{g}$, if $g \equiv 3 \pmod{5}$

\item $\mathcal E(g):y^{\frac{5}{2}g}=x(x+1)^{\frac{3g}{2}}$, if $g \equiv 2 \pmod{5}$

\item $\mathcal F(g):y^{\frac{5}{2}g}=x(x+1)^{2g}$, if $g \equiv 4 \pmod{5}$

\end{itemize}

\end{thm}
\begin{proof}
The case $|H|=5g$ corresponds to $k=5$ in Theorem \ref{tame-normal-form}. Therefore, $g\geq 5$ and the possibilities for $v$ in Equation \ref{esplicita} are $v\in \{1,\ldots,4\}$ which correspond to the curves in the statement. For each curve, the value of $g \pmod 5$ is determined by the condition $\gcd(N,s+2)\geq 3$. Moreover, $k-1=4$ being a divisor of $2g$ implies $g$ even, while $\gcd(k,2g)\leq 2$ yields $g\not\equiv 0\pmod{5}$. 
To conclude the proof, we need to show that each curve $\mathcal C(g),\ldots,\mathcal F(g)$ has an automorphism group of order $5g$ admitting a cyclic group of index $2$. This follows from Theorem \ref{suffspettro}.
\end{proof}

\begin{thm}
If $|H|<5g$, then $|H|\le \frac{24}{5}g$. Also, $|H|=\frac{24}{5}g$ if and only if $g\geq 10$, $g\equiv 0 \pmod 5$, $g\not\equiv 0 \pmod 3$, and $\cX$ is birationally equivalent to one of the curves
\begin{itemize}
\item $\mathcal L(g):y^{\frac{12}{5}g}=x(x+1)^\frac{2}{5}g,\,$ and  $g\equiv 1 \pmod 3$;
\item $\mathcal M(g):y^{\frac{12}{5}g}=x(x+1)^{2g},\, $ and $g\equiv 2 \pmod 3$.
\end{itemize}
\end{thm}
\begin{proof}
The case $|H|=24/5g$ corresponds to $k=6$ in Theorem \ref{tame-normal-form}. Therefore, $g\geq 6$ and the possibilities for $v$ in Equation \ref{esplicita} are $v=1,5$ which correspond to the curves $\mathcal L(g)$ and $\mathcal M(g)$ in the statement. 
As $k-1=5$ is a divisor of $2g$ we have $g\equiv 0\pmod 5$, while $\gcd(k,2g)\leq 2$ yields $g\not\equiv 0\pmod{3}$. In particular, $g\geq 10$. Now, let $g=5l$ for some $l\geq 2$. If $v=1$ then $s=2/5g$, and $\gcd(N,s+2)\geq 3$ yields
$\gcd(12l,2l+2)\geq 3$, that is either $2\mid l+1$ or $3\mid l+1$. 
In the latter case, $g\equiv 1\pmod 3$. In the former case, as $\gcd(N,s+1)=1$, we have that $3$ does not divide $2l+1$. Therefore, $l\not\equiv 1 \pmod 3$ and $g\not\equiv 2\pmod 3$. As $g\not\equiv 0\pmod 3$, $g\equiv 1\pmod 3$ holds.
With the same approach, if $v=5$ then $s=2g$ and $\gcd(N,s+2)\geq 3$ yields that either $2\mid l-1$ or $3\mid l-1$. In the latter case $g\equiv 2\pmod 3$ holds. As $\gcd(N,s+1)=1$, arguing as for $v=1$ we obtain $l\not\equiv 2\pmod 3$ and hence $g\equiv 2\pmod 3$.
To conclude the proof, we need to show that the curves $\mathcal L(g)$ and $\mathcal M(g)$ have an automorphism group of order $\frac{24}{5}g$ admitting a cyclic subgroup of index $2$. This follows from Theorem \ref{suffspettro}.
\end{proof}

\begin{rem} By Proposition \ref{espliciti}, for all curves listed above, the group $H$ has the following abstract form:
$$
H=\langle\rho,\sigma \mid \rho^N=\sigma^2=1,\quad \sigma\rho\sigma=\rho^{N-s-1}\rangle.
$$
\end{rem}


\section*{Acknowledgements}
The authors thank the Italian National Group for Algebraic and Geometric Structures and their Applications (GNSAGA—INdAM)
which supported the research. 

\end{document}